\documentclass{amsart}

\usepackage[font=small,labelfont=bf,up,textfont=it,up]{caption}

\usepackage[english]{babel}
\usepackage{tikz,graphicx}
\usetikzlibrary{matrix,arrows}

\usepackage{amsmath, amssymb, amsthm}

\newtheorem{theorem}{Theorem}[section]

\newtheorem{cor}[theorem]{Corollary}
\theoremstyle{definition}
\newtheorem{definition}[theorem]{Definition}
\newtheorem{remark}[theorem]{Remark}

\newcommand{\qq}{\mathbb{Q}}

\newcommand{\cc}{\mathbb{C}}

\newcommand{\zz}{\mathbb{Z}}

\makeatletter
\@namedef{subjclassname@2020}{\textup{2020} Mathematics Subject Classification}
\makeatother

\begin{document}

\author{Heiko Knospe}
\title{A note on regularized Bernoulli distributions and $p$-adic Dirichlet expansions}

\email{ heiko.knospe@th-koeln.de} 


\begin{abstract}
We consider Bernoulli distributions and their regularizations, which are measures on the $p$-adic integers $\zz_p$. It is well known that their Mellin transform can be used to define 
$p$-adic $L$-functions. We show that for $p>2$ one of the regularized Bernoulli distributions is particularly simple and equal to a measure on $\zz_p$ that takes the values $\pm \frac{1}{2}$ on clopen balls. We apply this to $p$-adic $L$-functions for Dirichlet characters of $p$-power conductor and obtain Dirichlet series expansions similar to the complex case. Such expansions were studied by D. Delbourgo, and this contribution provides an approach via $p$-adic measures.
\end{abstract}

\maketitle

\section{Introduction}

$p$-adic $L$-functions are $p$-adic analogues of complex $L$-functions. They have a long history and the primary constructions going back to by Kubota-Leopoldt \cite{kubota} and Iwasawa \cite{iwasawa} are via the interpolation of special values  of $L$-functions. 
$p$-adic $L$-functions can be defined via the unbounded Haar {\em distribution} on $\zz_p^*$, which yields a Volkenborn integral. However, it is also possible to use $p$-adic {\em measures}.

Let $p>2$ be a prime number and $k \geq 0$ an integer. For $x \in \zz_p$, we denote by $\{x\}_{p^n}$ the unique representative of $x \mod p^n$ between $0$ and $p^n-1$.

The {\em Bernoulli distributions} $E_k$ on $\zz_p$ are defined by 
$$ E_k( a + p^ n \zz_p) = p^{n(k-1)} B_k \left(\frac{ \{ a \}_{p^n} } {p^n} \right) , $$
where $B_k(x)$ is the $k$-th Bernoulli polynomial and $B_k=B_k(0)$ are the Bernoulli numbers (see \cite{koblitz}). For $k=0$, $B_0 (x) = 1$ and $E_0$ is the Haar distribution. For $k=1$, one has  $B_1(x)= x - \frac{1}{2}$.
Choose $c \in \zz$ with $c \neq 1$ and $c \notin p\zz$.     
Then the {\em regularization} of $E_k$ is defined by
$$ E_{k,c} ( a + p^n \zz_p ) = E_k ( a + p^n \zz_p) - c^k E_k \left( \left\{ \frac{a}{c}  \right\}_{p^n} + p^n \zz_p \right) .$$
One shows that the regularized Bernoulli distributions $E_{k,c}$ are {\em measures}. 
Now let $\chi: \zz \rightarrow \overline{\qq} \subset \cc_p$ be a Dirichlet character of conductor $f_{\chi}=p^m$, $m\geq 0$.
We denote by $\omega$ the Teichm\"uller character modulo $p$. For $a \in \zz^*_p$, we write $\langle a \rangle = \frac{a}{\omega(a)} \in 1+ p\zz_p$. 

It is well known that the $p$-adic $L$-function $L_p(s,\chi)$ is   a Mellin-transform of $E_{1,c}$ (see \cite{washington} Theorem 12.2):
$$ L_p(s, \chi) = \frac{-1}{1- \chi(c) \langle c \rangle^{-s+1}} \int_{\zz_p^*} \chi \omega^{-1} (a) \langle a \rangle^{-s} d E_{1,c} $$
$L_p(s,\chi)$ interpolates the complex values at $s=0,-1,-2,\dots$ up to a factor, i.e., for integers $k\geq 1$ (see \cite{washington} Theorem 5.11):
$$ L_p(1-k,\chi) = -(1-\chi\omega^{-k} (p) p^{k-1}) \frac{B_{k,\chi \omega^{-k}}}{k} $$

\section{Choosing a $p$-adic measure}

The regularization of the Bernoulli distributions depends on a parameter $c$ and there seems to be no preferred choice. However, we will show below that $E_{1,c}$ is particularly simple for $c=2$.

\begin{definition} Let $ p \neq 2$ be a prime. Then 
$$ \mu(a + p^n \zz_p) = (-1)^{\{a\}_{p^n}} $$
defines a normalized measure on $\zz_p$. We call $\mu$ the {\em alternating measure}, since the value on all clopen balls is $\pm 1$.\hfill$\lozenge$ 
\end{definition}
It is very easy to verify that $\mu$ is in fact a measure. Note that we excluded $p=2$. The Theorem below shows that $\mu$ is (up to the factor $\frac{1}{2})$ equal to the regularized Bernoulli distribution $E_{1,2}$.

\begin{theorem} Let $p>2$ and $\mu$ the above alternating measure. Then $$\textstyle\frac{1}{2} \mu = E_{1,2}.$$
\label{measure}
\end{theorem}

\begin{proof} Let $n \geq 1$. First, we consider even representatives of $a + p^n \zz_p$.
Let $a=2b$ where $b \in \{0,1, \dots , \frac{p^n-1}{2} \}$. Then 
$$E_{1,2}(a+p^n \zz_p) = E_1(2b + p^n \zz_p) - 2 E_1(b + p^n \zz_p) = \frac{2b}{p^n} - \frac{1}{2} - 2 \left( \frac{b}{p^n} - \frac{1}{2} \right) = \frac{1}{2} .$$
Now we look at odd representatives. Let $a=2b+1$ where $b \in \{0,1, \dots , \frac{p^n-3}{2} \}$. Then we have
$$E_{1,2}(a+p^n \zz_p) = E_1(2b+1 + p^n \zz_p) - 2 E_1(b + \textstyle\frac{1}{2} + p^n \zz_p) = \frac{2b+1}{p^n} -\frac{1}{2} - 2 \left(\frac{ \{ b + \frac{1}{2} \}_{p^n} } {p^n} - \frac{1}{2} \right) .$$
Note that $\frac{1}{2}$ has the following $p$-adic representation:
$$ \frac{1}{2} = \frac{p+1}{2} +  \frac{p-1}{2} p + \frac{p-1}{2} p^2 + \dots $$ 

Hence $\{ \frac{1}{2} \}_{p^n} = \frac{p^n+1}{2}$. Since $b=\{b\}_{p^n} \leq \frac{p^n - 3}{2}$, we have $\{ \frac{1}{2} \}_{p^n} + \{b\}_{p^n} \leq p^n-1$. So 
we can add the representatives and obtain
$$ \left\{ b + \frac{1}{2} \right\}_{p^n} = b + \frac{p^n+1}{2} .$$
This implies our assertion:
$$E_{1,2}(a+p^n \zz_p) = \frac{2b+1}{p^n} -\frac{1}{2} - \frac{2b +p^n+1 }{p^n} + 1 = - \frac{1}{2} $$ 
 
\end{proof}

\section{Dirichlet series expansion}
We can use Theorem \ref{measure} to obtain a $p$-adic Dirichlet series expansion of $L_p(s,\chi)$.

\begin{cor} Let $p>2$ be a prime and $\chi$ a Dirichlet character of $p$-power conductor. Then:
$$  L_p(s, \chi) = \frac{-1}{1- \chi(2) \langle 2 \rangle^{-s+1}} \cdot  \lim_{n \rightarrow \infty} \sum_{\substack{a=1 \\ p\, \nmid\, a}}^{p^n} \frac{(-1)^a}{2} \, \chi \omega^{-1} (a) \langle a \rangle^{-s}  $$
$L_p(s,\chi)$ is analytic in $s \in \zz_p$, except a simple pole at $s=1$ for $\chi=\omega$. For $\chi=\omega^i$ and $i=0,\, 1,\, \dots,\, p-2$, we obtain the $p-1$ branches of the $p$-adic zeta function:
$$ \zeta_{p,i}(s) = L_p(s,\, \omega^{1-i}) = \frac{-1}{1-\omega(2)^{1-i} \langle 2 \rangle^{-s+1}}  \cdot  \lim_{n \rightarrow \infty} \sum_{\substack{a=1 \\ p\, \nmid\, a }}^{p^n} \frac{(-1)^a}{2} \ \omega(a)^{-i} \langle a \rangle^{-s} $$
\end{cor}

\begin{remark}
The integral representation of $p$-adic $L$-functions using measures  and Iwasawa's construction using Stickelberger elements (see \cite{iwasawa}) suggest  that a Dirichlet series expansion is possible. However, the exact coefficients of $\langle a \rangle^{-s}$ are not obvious, even for the $p$-adic zeta function.
We see that the complex and the $p$-adic expansions are surprisingly similar. In the $p$-adic case, we have to look at the above subsequence of the series since $|\langle a \rangle^{-s}|_p=1$.

Dirichlet series expansions were studied by D. Delbourgo  in \cite{delbourgo2006dirichlet},  \cite{delbourgo2009}. 
He considers Dirichlet characters $\chi$ with $\gcd(p, 2 f_{\chi} \phi(f_{\chi}))=1$ and their Teichm\"uller twists. 
 We obtain similar results, but consider the case $f_{\chi} = p^m$ and use other methods ($p$-adic measures).  For the $p-1$ branches of the $p$-adic zeta function, we obtain the same expansions. The values $\mp \frac{1}{2}$ of the scaled alternating measure $E_{1,2}= \frac{1}{2} \mu$ can be found as coefficients $a_1(\chi)$ and $a_2(\chi)$ in the first row of Table 1 in \cite{delbourgo2009}.
 \end{remark}

\section*{Acknowledgments}
The author wishes to thank Daniel Delbourgo for hints to his work and helpful conversations.

\bibliography{nsabib}

\bibliographystyle{plain}

\end{document}